\documentclass{article}
\usepackage{arxiv}
\usepackage[toc,page]{appendix}

\usepackage[T1]{fontenc}    %
\usepackage{hyperref}       %
\usepackage{url}            %
\usepackage{booktabs}       %
\usepackage{amsfonts}       %
\usepackage{nicefrac}       %
\usepackage{microtype}      %
\usepackage{lipsum}
\usepackage{mathtools}
\usepackage{cite}
\usepackage{amsmath}
\usepackage{amssymb}
\usepackage{amsthm}
\usepackage{algorithm,algcompatible}
\usepackage[toc,page]{appendix}

\usepackage{lineno}

\newtheorem{theorem}{Theorem}
\newtheorem{lemma}{Lemma}

\newtheorem{corollary}{Corollary}
\newtheorem{proposition}{Proposition}

\usepackage{mathptmx}      %
\begin{document}

\title{Optimal Convergence Rates for the Orthogonal Greedy Algorithm
}

\author{Jonathan W. Siegel \\
  Department of Mathematics\\
  Pennsylvania State University\\
  University Park, PA 16802 \\
  \texttt{jus1949@psu.edu} \\
  \And Jinchao Xu \\
  Department of Mathematics\\
  Pennsylvania State University\\
  University Park, PA 16802 \\
  \texttt{jxx1@psu.edu} \\
}

\maketitle

\begin{abstract}
We analyze the orthogonal greedy algorithm when applied to dictionaries $\mathbb{D}$ whose convex hull has small entropy. We show that if the metric entropy of the convex hull of $\mathbb{D}$ decays at a rate of $O(n^{-\frac{1}{2}-\alpha})$ for $\alpha > 0$, then the orthogonal greedy algorithm converges at the same rate on the variation space of $\mathbb{D}$.
This improves upon the well-known $O(n^{-\frac{1}{2}})$ convergence rate of the orthogonal greedy algorithm in many cases, most notably for dictionaries corresponding to shallow neural networks. These results hold under no additional assumptions on the dictionary beyond the decay rate of the entropy of its convex hull. In addition, they are robust to noise in the target function and can be extended to convergence rates on the interpolation spaces of the variation norm. We  show empirically that the predicted rates are obtained for the dictionary corresponding to shallow neural networks with Heaviside activation function in two dimensions. Finally, we show that these improved rates are sharp  and prove a negative result showing that the iterates generated by the orthogonal greedy algorithm cannot in general be bounded in the variation norm of $\mathbb{D}$.
\end{abstract}

\section{Introduction}
Let $H$ be a Hilbert space and $\mathbb{D}\subset H$ a dictionary of basis functions. An important problem in machine learning, statistics, and signal processing is the non-linear approximation of a target function $f$ by a sparse linear combinations of dictionary elements
\begin{equation}\label{sparse-dictionary expansion}
 f_n = \sum_{i=1}^n a_ig_i,
\end{equation}
where $g_i\in \mathbb{D}$ depend upon the function $f$. Typical examples include non-linear approximation by redundant wavelet frames \cite{mallat1993matching}, shallow neural networks, which correspond to non-linear approximation by dictionaries of ridge functions \cite{lee1996efficient}, and gradient boosting \cite{friedman2001greedy}, which corresponds to non-linear approximation by a dictionary of weak learners. Another important example is compressed sensing \cite{donoho2006compressed,candes2006robust,cohen2009compressed}, where a function $f$ which is a sparse linear combination of dictionary elements is recovered from a small number of linear measurements.

In this work, we study the problem of algorithmically calculating an expansion of the form \eqref{sparse-dictionary expansion} to approximate a target function $f$. A common class of algorithms for this purpose are greedy algorithms, specifically the pure greedy algorithm \cite{mallat1993matching},
\begin{equation}
 f_0 = 0,~g_k = \arg\max_{g\in \mathbb{D}}|\langle g,f - f_{k-1}\rangle|,~f_k = f_{k-1} + \langle g_k,f - f_{k-1}\rangle g_k,
\end{equation}
which is also known as matching pursuit, the relaxed greedy algorithm \cite{jones1992simple,barron2008approximation,barron1993universal},
\begin{equation}\label{relaxed-greedy-algorithm}
 f_0 = 0,~(\alpha_k,\beta_k,g_k) = \arg\min_{\alpha_k,\beta_k\in \mathbb{R},g_k\in \mathbb{D}} \|f - \alpha_kf_{k-1} - \beta_kg_k\|_H,~f_k = \alpha_kf_{k-1} + \beta_kg_k,
\end{equation}
and the orthogonal greedy algorithm (also known as orthogonal matching pursuit) \cite{pati1993orthogonal}
\begin{equation}\label{orthogonal-greedy}
 f_0 = 0,~g_k = \arg\max_{g\in \mathbb{D}} |\langle g,r_{k-1}\rangle|,~f_k = P_kf,
\end{equation}
where $r_k = f - f_k$ is the residual and $P_k$ is the orthogonal projection onto the span of $g_1,...,g_k$. 

When applied to general dictionaries $\mathbb{D}\subset H$, greedy algorithms are often analyzed under the assumption that the target function $f$ lies in the convex hull of the dictionary $\mathbb{D}$. Specifically, following \cite{siegel2021optimal,devore1998nonlinear,kurkova2001bounds} we write the closed convex hull of $\mathbb{D}$ as
\begin{equation}\label{unit-ball-definition}
 B_1(\mathbb{D}) = \overline{\left\{\sum_{j=1}^n a_jh_j:~n\in \mathbb{N},~h_j\in \mathbb{D},~\sum_{i=1}^n|a_i|\leq 1\right\}},
\end{equation}
and denote the gauge norm \cite{rockafellar1970convex} of this set, which is often called the variation norm with respect to $\mathbb{D}$, as
\begin{equation}\label{norm-definition}
 \|f\|_{\mathcal{K}_1(\mathbb{D})} = \inf\{c > 0:~f\in cB_1(\mathbb{D})\}.
\end{equation}
A typical assumption then is that the target function $f$ satisfies $\|f\|_{\mathcal{K}_1(\mathbb{D})} < \infty$. One notable exception is compressed sensing, where it is typically assumed that $f$ is a linear combination of a small number of elements $g_i\in \mathbb{D}$. In this case, however, one needs an additional assumption on the dictionary $\mathbb{D}$, usually incoherence \cite{donoho2005stable} or a restricted isometry property (RIP) \cite{candes2005decoding}. The analysis of greedy algorithms under these assumptions is given in \cite{tropp2007signal,needell2009cosamp,needell2009uniform,tropp2004greed} for instance. In this work we are interested in the case of general dictionaries $\mathbb{D}\subset H$ which do not satisfy any incoherence or RIP conditions, however, so we will consider the convex hull condition on $f$ given above.

Given a function $f\in B_1(\mathbb{D})$, a sampling argument due to Maurey \cite{pisier1981remarques} implies that there exists an $n$-term expansion which satisfies
\begin{equation}\label{maurey-bound}
 \inf_{f_n\in \Sigma_n(\mathbb{D})} \|f - f_n\|_H \leq |\mathbb{D}|n^{-\frac{1}{2}}.
\end{equation}
Here $|\mathbb{D}| := \sup_{g\in \mathbb{D}}\|g\|_H$ is the maximum norm of the dictionary and
\begin{equation}
 \Sigma_n(\mathbb{D}) = \left\{\sum_{i=1}^n a_ig_i,~a_i\in \mathbb{R},~g_i\in \mathbb{D}\right\}
\end{equation}
is the set of $n$-term non-linear dictionary expansions. One can even take the expansion $f_n$ in \eqref{maurey-bound} to satisfy $\sum_{i=1}^n |a_i| \leq 1$. It is also known that for general dictionaries $\mathbb{D}$, the approximation rate \eqref{maurey-bound} is the best possible \cite{kurkova2001bounds} up to a constant factor. The key problem addressed by greedy algorithms is whether the rate in \eqref{maurey-bound} can be achieved algorithmically.

For simplicity of notation, we assume in the following that $|\mathbb{D}| \leq 1$, i.e. that $\|g\|_H\leq 1$ for all $g\in \mathbb{D}$, which can always be achieved by scaling the dictionary appropriately. Relaxing this assumption changes the results in a straightforward manner. It has been shown that under these assumptions the orthogonal greedy algorithm \cite{devore1996some} and a suitable version of the relaxed greedy algorithm \cite{jones1992simple,barron2008approximation} satisfy
\begin{equation}\label{greedy-initial-approx}
 \|f_n - f\|_H \leq K\|f\|_{\mathcal{K}_1(\mathbb{D})}n^{-\frac{1}{2}},
\end{equation}
for a suitable constant $K$ (here and in the following $K$, $M$, and $C$ represent unspecified constants). Thus the orthogonal and relaxed greedy algorithms are able to algorithmically attain the approximation rate in \eqref{maurey-bound} up to a constant. The behavior of the pure greedy algorithm is much more subtle. A sequence of improved upper bounds on the convergence rate of the pure greedy algorithm have been obtained in \cite{devore1996some,livshits2004rate,sil2004rate}, which culminate in the bound
\begin{equation}
 \|f_n - f\|_H \leq K\|f\|_{\mathcal{K}_1(\mathbb{D})}n^{-\gamma},
\end{equation}
where $\gamma \approx 0.182$ satisfies a particular non-linear equation. Conversely, in \cite{livshits2009lower} a dictionary is constructed for which the pure greedy algorithm satisfies
\begin{equation}
 \|f_n - f\|_H \geq K\|f\|_{\mathcal{K}_1(\mathbb{D})}n^{-0.1898}.
\end{equation}
Thus the precise convergence rate of the pure greedy algorithm is still open, but it is known that it fails to achieve the rate \eqref{maurey-bound} by a significant margin. A further interesting open problem concerns the pure greedy algorithm with shrinkage $s\in (0,1]$, which is given by
\begin{equation}
 f_0 = 0,~g_k = \arg\max_{g\in \mathbb{D}}|\langle g,r_{k-1}\rangle|,~f_k = f_{k-1} + s\langle g_k,r_{k-1}\rangle g_k,
\end{equation}
where again $r_k = f - f_k$ is the residual.
This algorithm is important for understanding gradient boosting \cite{friedman2001greedy}. It is shown in \cite{sil2004rate} that the convergence order improves as $s$ decreases to a maximum of about $0.305$ as $s\rightarrow 0$, but it is an open problem whether the optimal rate \eqref{maurey-bound} can be achieved in the limit as $s\rightarrow 0$.

The preceding results hold under the assumption that $\|f\|_{\mathcal{K}_1(\mathbb{D})} < \infty$, i.e. that the target function is in the scaled convex hull of the dictionary. An important question is how robust the greedy algorithms are to noise. This is captured by the more general assumption that $f$ is not itself in the convex hull, but rather that $f$ is close to an element in the convex hull. Specifically, we assume that $h\in \mathcal{K}_1(\mathbb{D})$ and that $\|f - h\|_H$, which can be thought of as noise, is small. The convergence rate of the orthogonal and relaxed greedy algorithms are analyzed under these assumptions in \cite{barron2008approximation}. Specifically, it is shown that for any $f\in H$ and any $h\in \mathcal{K}_1(\mathbb{D})$, the iterates of the orthogonal greedy algorithm satisfy
\begin{equation}
 \|f_n - f\|^2_H \leq \|f - h\|_H^2 + K^2\|h\|^2_{\mathcal{K}_1(\mathbb{D})}n^{-1}.
\end{equation}
This result is important for the statistical analysis of the orthogonal greedy algorithm, for instance in showing the universal consistency of the estimator obtained by applying this algorithm to fit a function $f$ based on samples $f(x_1),...,f(x_n)$ \cite{barron2008approximation}. It also gives a convergence rate for the orthogonal greedy algorithm on the interpolation spaces between $\mathcal{K}_1(\mathbb{D})$ and $H$. The $K$-functional of $H$ and $\mathcal{K}_1(\mathbb{D})$ is defined by
\begin{equation}
 K(f,t) := K(f,t,H,\mathcal{K}_1(\mathbb{D})) = \inf_{g\in H} \|f - g\|_H + t\|g\|_{\mathcal{K}_1(\mathbb{D})}.
\end{equation}
The interpolation space $X_\theta := [H,\mathcal{K}_1(\mathbb{D})]_{\theta,\infty}$ is defined by the norm $\|f\|_X = \sup_{t > 0} K(f,t)t^{-\theta}$. Intuitively, the interpolation spaces measure how efficiently a function can be approximated by elements with small $\mathcal{K}_1(\mathbb{D})$-norm.
We refer to \cite{barron2008approximation} and \cite{devore1993constructive}, Chapter 6, for more information on interpolation spaces.

Although the approximation rate \eqref{maurey-bound} is sharp for general dictionaries, for compact dictionaries the rate can be improved using a stratified sampling argument \cite{makovoz1996random,klusowski2018approximation,CiCP-28-1707}. Specifically, the dictionary
\begin{equation}
 \mathbb{P}_0^d = \{\sigma_0(\omega\cdot x + b),~\omega\in \mathbb{R}^d,~b\in \mathbb{R}\}\subset L^2(\Omega),
\end{equation}
where the activation function $\sigma_0$ is the Heaviside function and $\Omega\subset \mathbb{R}^d$ a bounded domain was analyzed in \cite{makovoz1996random}. There it was shown that for $f\in B_1(\mathbb{P}_0^d)$ the approximation rate \eqref{maurey-bound} can be improved to
\begin{equation}
 \inf_{f_n\in \Sigma_n(\mathbb{P}_0^d)} \|f - f_n\|_H \leq Kn^{-\frac{1}{2}-\frac{1}{2d}}.
\end{equation}
This corresponds to an improved approximation rate for shallow neural networks with Heaviside (and also sigmoidal) activation function. For shallow neural networks with ReLU$^k$ activation function, the relevant dictionary \cite{siegel2021sharp} is
\begin{equation}
 \mathbb{P}_k^d = \{\sigma_k(\omega\cdot x + b),~\omega\in S^{d-1},~b\in [-c,c]\}\subset L^2(\Omega),
\end{equation}
where $S^{d-1}$ is the unit sphere is $\mathbb{R}^d$ and $c$ is the diameter of $\Omega$ (assuming without loss of generality that $0\in \Omega$). Using the smoothness of the dictionary $\mathbb{P}_k^d$, it has been shown that for $f\in B_1(\mathbb{P}_k^d)$ an approximation rate of
\begin{equation}\label{eq-189}
 \inf_{f_n\in \Sigma_n(\mathbb{P}_k^d)} \|f - f_n\|_H \leq Kn^{-\frac{1}{2}-\frac{2k+1}{2d}}
\end{equation}
can be achieved, and that this rate is optimal if the coefficients of $f_n$ are bounded \cite{siegel2021sharp}. Given these improved theoretical approximation rates for shallow neural networks, it has been an important open problem whether they can be achieved algorithmically by greedy algorithms.

In this work, we show that the improved approximation rates \eqref{eq-189} can be achieved using the orthogonal greedy algorithm. More generally, we show that the orthogonal greedy algorithm improves upon the rate \eqref{maurey-bound} whenever the convex hull of the dictionary $\mathbb{D}$ has small metric entropy. Specifically, we recall that the (dyadic) metric entropy of a set $A$ in a Banach space $X$ is defined by
\begin{equation}
 \epsilon_n(A)_X = \inf\{\epsilon:~A~\text{can be covered by $2^n$ balls of radius $\epsilon$}\}.
\end{equation}
The metric entropy gives a measure of compactness of the set $A$ and a detailed theory can be found in \cite{lorentz1996constructive}, Chapter 15. We show that if the dictionary $\mathbb{D}$ satisfies
\begin{equation}\label{entropy-condition}
 \epsilon_n(B_1(\mathbb{D})_H \leq Kn^{-\frac{1}{2}-\alpha},
\end{equation}
then the orthogonal greedy algorithm \eqref{orthogonal-greedy} satisfies
\begin{equation}\label{improved-rate}
 \|f_n - f\|_H \leq K'\|f\|_{\mathcal{K}_1(\mathbb{D})}n^{-\frac{1}{2}-\alpha},
\end{equation}
where $K'$ only depends upon $K$ and $\alpha$. More generally, this analysis is also robust to noise in the sense considered in \cite{barron2008approximation}, i.e. for any $f\in H$ and $g\in \mathcal{K}_1(\mathbb{D})$ we have
\begin{equation}
 \|f_n - f\|^2_H \leq \|f - h\|_H^2 + (K')^2\|h\|^2_{\mathcal{K}_1(\mathbb{D})}n^{-1-2\alpha}.
\end{equation}
Utilizing the metric entropy bounds proven in \cite{siegel2021sharp}, this implies that the orthogonal greedy algorithm achieves the rate \eqref{eq-189} for shallow neural networks with ReLU$^k$ activation function. We provide numerical experiments using shallow neural networks with the Heaviside activation function which confirm that these optimal rates are indeed achieved. Additional numerical experiments which confirm the theoretical approximation rates for shallow networks with ReLU$^k$ activation function can be found in \cite{hao2021efficient}, where the orthogonal greedy algorithm is used to solve elliptic PDEs.

We conclude the manuscript with a lower bound and negative result concerning the orthogonal greedy algorithm. Consider approximating $f\in B_1(\mathbb{D})$ by dictionary expansions with $\ell^1$-bounded coefficients, i.e. from the set
\begin{equation}
 \Sigma_{n,M}(\mathbb{D}) = \left\{\sum_{j=1}^n a_jh_j:~h_j\in \mathbb{D},~\sum_{i=1}^n|a_i|\leq M\right\}.
\end{equation}
It is known that the corresponding approximation rates are lower bounded by the metric entropy up to logarithmic factors. In particular, if for some $M > 0$ and any $f\in B_1(\mathbb{D})$, we have
\begin{equation}
 \inf_{f_n \in \Sigma_{n,M}(\mathbb{D})} \|f - f_n\| \leq Kn^{-\frac{1}{2}-\alpha},
\end{equation}
then we must have $\epsilon_{n\log{n}}(B_1(\mathbb{D}))\leq Kn^{-\frac{1}{2}-\alpha}$. The result \eqref{improved-rate} shows that the orthogonal greedy algorithm achieves this rate. 
However, we note that these lower bounds do not a priori apply to the orthogonal greedy algorithm since the expansions generated will in general not have coefficients uniformly bounded in $\ell^1$. In fact, we give an example demonstrating that the iterates generated by the orthogonal greedy algorithm may have arbitrarily large $\mathcal{K}_1(\mathbb{D})$-norm. Nonetheless, we will show that the rate \eqref{improved-rate} is sharp and cannot be further improved for general dictionaries satisfying \ref{entropy-condition}. We do not know whether there exists a greedy algorithm which can attain the improved approximation rate \eqref{improved-rate} while also guaranteeing a uniform $\ell^1$-bound on the coefficients.

The paper is organized as follows. In the Section \ref{orthogonal-greedy-section}, we derive the improved convergence rate \eqref{improved-rate}. Next, in Section \ref{numerical-experiments-section} we provide numerical experiments which demonstrate the improved rates. Then, in Section \ref{lower-bounds-section}, we give an example showing that the iterates generated by the orthogonal greedy algorithm cannot be bounded in $\mathcal{K}_1(\mathbb{D})$ and also show that the improved rates are tight. Finally, we give concluding remarks and further research directions.

\section{Analysis of the Orthogonal Greedy Algorithm}\label{orthogonal-greedy-section}
We begin with the following key lemma.
\begin{lemma}\label{entropy-packing-lemma}
 Let $\delta > 0$ and $\mathbb{D}$ be a dictionary with
 \begin{equation}\label{entropy-lemma-assumption}
  \epsilon_n(B_1(\mathbb{D})) \leq Cn^{-\frac{1}{2} - \delta}.
 \end{equation}
 Then there exists a $c = c(\delta, C) > 0$ such that for any sequence $d_1,...,d_n\in \mathbb{D}$, we have
  \begin{equation}\label{lemma-1-conclusion}
  \sum_{k=1}^n \|(I - P_{k-1})d_k\|^{-2} \geq cn^{1+2\delta},
 \end{equation}
 where $P_k$ is the orthogonal projection onto the span of $d_1,...,d_k$.
\end{lemma}
The intuituve idea behind this lemma is that if \eqref{lemma-1-conclusion} fails, then the convex hull $B_1(\mathbb{D})$ must contain a relatively large skewed simplex. A comparison of volumes then gives a lower bound on its entropy.

\begin{proof}
 By scaling the dictionary, we may assume without loss of generality that $C = 1$. Fix a $c > 0$ and consider a sequence $d_1,...,d_{2n}\in \mathbb{D}$. Assume that
 \begin{equation}\label{contradicted-assumption}
  \sum_{k=1}^{2n} \|(I - P_{k-1})d_k\|^{-2} \leq cn^{1+2\delta}.
 \end{equation}
 We will show that for sufficiently small $c$ this contradicts the entropy bound \eqref{entropy-lemma-assumption}. Rescaling this value of $c$ by $2^{-1-2\delta}$ gives the desired bound \eqref{lemma-1-conclusion}.
 
 Choose the $n$ indices $k_1 < 
 \cdots < k_n$ for which $\|(I - P_{k-1})d_k\|^{-2}$ is the smallest. For each $i=1,...,n$, there are at least $n$ indices $k$ for which $\|(I - P_{k-1})d_k\|^{-2} \geq \|(I - P_{k_i-1})d_{k_i}\|^{-2}$. This gives the inequality
 \begin{equation}
  \sum_{k=1}^{2n} \|(I - P_{k-1})d_k\|^{-2} \geq n\|(I - P_{k_i-1})d_{k_i}\|^{-2},
 \end{equation}
 so that by \eqref{contradicted-assumption} we must have $\|(I - P_{k_i-1})d_{k_i}\|^{-2} \leq cn^{2\delta}$ and thus 
 \begin{equation}\label{eq-2031}
  \|(I - P_{k_i-1})d_{k_i}\| \geq c^{-\frac{1}{2}}n^{-\delta}
 \end{equation}
 for each $i=1,...,n$.
 Further, if we replace the projections $P_{k_i-1}$ by the orthogonal projection onto the span of $d_{k_1},...,d_{k_{i-1}}$, the length in \eqref{eq-2031} can only increase (since we are removing dictionary elements from the projection). Thus, by relabelling we obtain a sequence $d_1,...,d_n\in \mathbb{D}$ such that
 \begin{equation}\label{eq-2035}
  \|(I - P_{k-1})d_{k}\| \geq c^{-\frac{1}{2}}n^{-\delta}
 \end{equation}
 for $k=1,...,n$.
 
 Let $\tilde{d}_1,...,\tilde{d}_n$ be the Gram-Schmidt orthogonalization of the sequence $d_1,...,d_n$, i.e. $\tilde{d}_k = (I - P_{k-1})d_{k}$. From \eqref{eq-2035}, we obtain $\|\tilde{d}_i\| \geq c^{-\frac{1}{2}}n^{-\delta}$ for each $i=1,...,n$, and thus since the $\tilde{d}_i$ are orthogonal, we obtain the following bound
 \begin{equation}
  |\text{co}(\pm\tilde{d}_1,...,\pm\tilde{d}_n)| \geq (c^{-\frac{1}{2}}n^{-\delta})^n\frac{2^n}{n!}.
 \end{equation}
 Here $\text{co}(\pm\tilde{d}_1,...,\pm\tilde{d}_n)$ is the absolute convex hull of the sequence $\tilde{d}_1,...,\tilde{d}_n$.
 
 Next we use the fact that the change of variables between the $\{d_i\}$ and the $\{\tilde{d}_i\}$ is upper triangular with ones on the diagonal and thus has determinant $1$ (since $P_{k-1}d_k\in \text{span}(d_1,...,d_{k-1})$). This implies that
 \begin{equation}
  |\text{co}(\pm d_1,...,\pm d_n)| = |\text{co}(\pm\tilde{d}_1,...,\pm\tilde{d}_n)| \geq (c^{-\frac{1}{2}}n^{-\delta})^n\frac{2^n}{n!}.
 \end{equation}
 In other words the convex hull of the $d_i$ is a skewed simplex with large volume.
 
 We now use the covering definition of the entropy, setting $\epsilon := \epsilon_n(B_1(\mathbb{D}))_H$, and the fact that $\text{co}(\pm d_1,...,\pm d_n)\subset B_1(\mathbb{D})$ to get
 \begin{equation}
  (c^{-\frac{1}{2}}n^{-\delta})^n\frac{2^n}{n!} \leq |\text{co}(\pm d_1,...,\pm d_n)| \leq |B_1(\mathbb{D})| \leq (2\epsilon)^n\frac{\pi^{n/2}}{\Gamma\left(\frac{n}{2}+1\right)},
 \end{equation}
 where the right hand side is the volume of $2^n$ balls of radius $\epsilon$.
 
 Utilizing Sterling's formula and taking $n$-th roots, we get
 \begin{equation}
  \epsilon \geq Kc^{-\frac{1}{2}}n^{-\frac{1}{2}-\delta},
 \end{equation}
 for an absolute constant $K$. For sufficiently small $c$, this will contradict the bound \eqref{entropy-lemma-assumption}, which completes the proof.
\end{proof}

Finally, we come to the main result of this section, which shows that the orthogonal greedy algorithm \eqref{orthogonal-greedy} achieves a convergence rate which matches the entropy for dictionaries $\mathbb{D}$ whose entropy decays faster than $O(n^{-\frac{1}{2}})$. In fact, we prove a bit more, generalizing the result from \cite{barron2008approximation} to obtain improved approximation rates for the interpolation spaces between $\mathcal{K}_1(\mathbb{D})$ and $H$ as well.

\begin{theorem}\label{main-theorem}
 Let $\delta > 0$ and suppose that $\mathbb{D}\subset H$ is a dictionary which satisfies
 \begin{equation}
  \epsilon_n(B_1(\mathbb{D})) \leq Cn^{-\frac{1}{2} - \delta},
 \end{equation}
 for some constant $C < \infty$.
 Let the iterates $f_n$ be given by algorithm \eqref{orthogonal-greedy}, where $f\in H$. Further, let $h\in \mathcal{K}_1(\mathbb{D})$ be arbitrary. Then we have
 \begin{equation}\label{interpolation-bound}
  \|f_n - f\|^2 \leq \|f - h\|^2 + K\|h\|^2_{\mathcal{K}_1(\mathbb{D})}n^{-1-2\delta},
 \end{equation}
 where $K$ depends only upon $\delta$ and $C$. In particular, if $f\in \mathcal{K}_1(\mathbb{D})$, then
 \begin{equation}
  \|f_n - f\| \leq Kn^{-\frac{1}{2} - \delta}.
 \end{equation}

\end{theorem}
We note that although $f_n$ is a linear combination of $n$ dictionary elements $g_1,...,g_n$, we do necessarily have that $f_n$ can be bounded in $\mathcal{K}_1(\mathbb{D})$, as we show in Proposition \ref{no-l1-bound-proposition}. This is due to the fact that the coefficients in the expansion generated by the orthogonal greedy algorithm \eqref{orthogonal-greedy} cannot be bounded in $\ell^1$.

Further, the bound \eqref{interpolation-bound} shows that the improved convergence behavior of the orthogonal greedy algorithm is robust to noise. This is critical for the statistical analysis of the algorithm, for instance for proving statistical consistency \cite{barron2008approximation}. Further, it enables one to prove a convergence rate for functions $f$ in the interpolation space $X_\theta = [H,\mathcal{K}_1(\mathbb{D})]_{\theta,\infty}$, analogous to the results in \cite{barron2008approximation}. We remark that in \cite{barron2008approximation} a Fourier integrability condition is given which guarantees membership in $X_\theta$ when $\mathbb{D} = \mathbb{P}_0^d$ is the dictionary corresponding to shallow neural networks with sigmoidal activation function. This generalizes the Fourier condition introduced by Barron in \cite{barron1993universal}.

\begin{corollary}
 Let $\mathbb{D}\subset H$ be a dictionary satisfying the assumptions of Theorem \ref{main-theorem}. Then for $f\in X_\theta$ we have
 \begin{equation}
  \|f_n - f\| \leq K\|f\|_{X_\theta}n^{-(\frac{1}{2} + \delta)\theta},
 \end{equation}
 where $K$ depends only upon $\delta$ and $C$.
\end{corollary}
\begin{proof}
 Taking the square root of \eqref{interpolation-bound} we see that
 \begin{equation}
  \|f_n - f\| \leq \|f - h\| + M\|h\|_{\mathcal{K}_1(\mathbb{D})}n^{-\frac{1}{2}-\delta}
 \end{equation}
 for a new constant $M$ (we can take $M = \sqrt{K}$). Taking the infemum over $h\in \mathcal{K}_1(\mathbb{D})$, we get
 \begin{equation}\label{eq-387}
  \|f_n - f\| \leq K(Mn^{-\frac{1}{2}-\delta}, f),
 \end{equation}
 where $K$ is the $K$-functional introduced in the introduction (see also \cite{devore1993constructive}, Chapter 6 for a more detailed theory). The definition of the interpolation space $X_\theta$ implies that
 \begin{equation}
  K(t,f) \leq \|f\|_{X_\theta}t^{\theta}
 \end{equation}
 for all $t > 0$. Setting $t = Mn^{-\frac{1}{2}-\delta}$ in \eqref{eq-387} gives the desired result.
\end{proof}

\begin{proof}[Proof of Theorem \ref{main-theorem}]
 Throughout the proof, let $P_k$ denote the projection onto $g_1,...,g_k$ and $r_k = f - f_k$ denote the residual. Since $f_k$ is the best approximation to $f$ from the space $V_k = \text{span}(g_1,...,g_k)$, we have
 \begin{equation}\label{eq-348}
  \|r_k\|^2 \leq \left\|r_{k-1} - \frac{\langle r_{k-1}, (I-P_{k-1})g_k\rangle}{\|(I-P_{k-1})g_k\|^{2}}(I-P_{k-1})g_k\right\|^2 = \|r_{k-1}\|^2 - \frac{|\langle r_{k-1}, (I-P_{k-1})g_k\rangle|^2}{\|(I-P_{k-1})g_k\|^{2}}.
 \end{equation}
 Next, we note that since $r_{k-1}$ is orthogonal to $g_1,...,g_{k-1}$ we have $\langle r_{k-1}, (I-P_{k-1})g_k\rangle = \langle r_{k-1}, g_k\rangle$. In addition, as $r_{k-1}$ is orthogonal to $f_{k-1} = f - r_{k-1}$, we see that
 \begin{equation}
 \begin{split}
 \|r_{k-1}\|^2 = \langle r_{k-1}, f\rangle &= \langle r_{k-1}, h + f - h\rangle\leq \|h\|_{\mathcal{K}_1}\max_{g\in \mathbb{D}} |\langle r_{k-1}, g\rangle|  + \langle r_{k-1}, f-h\rangle \\
 &\leq \|h\|_{\mathcal{K}_1} |\langle r_{k-1}, g_k\rangle|  + \frac{1}{2}\left(\|r_{k-1}\|^2 + \|f-h\|^2\right).
 \end{split}
 \end{equation}
 Setting $b_k = \|r_k\|^2 - \|f-h\|^2$, we can rewrite this as
 \begin{equation}
  b_{k-1} \leq 2\|h\|_{\mathcal{K}_1} |\langle r_{k-1}, g_k\rangle|,
 \end{equation}
 which gives the lower bound
 \begin{equation}
  |\langle r_{k-1}, (I-P_{k-1})g_k\rangle| = |\langle r_{k-1}, g_k\rangle| \geq b_{k-1}\|h\|_{\mathcal{K}_1}^{-1}.
 \end{equation}
 If $b_k$ is ever negative, then the desired result clearly holds for all $n \geq k$ since $b_k$ is decreasing. So we assume without loss of generality that $b_k > 0$ in what follows.
 
 Subtracting $\|f-h\|^2$ from equation \eqref{eq-348} and using lower bound above, we get the recursion
 \begin{equation}
  b_k \leq b_{k-1}\left(1 - \frac{b_{k-1}}{4\|(I-P_{k-1})g_k\|^2\|h\|_{\mathcal{K}_1}^2}\right).
 \end{equation}
 Define the sequence $a_k = (2\|h\|_{\mathcal{K}_1})^{-2}b_k$, and we get the recursion
 \begin{equation}\label{eq-329}
  a_k \leq a_{k-1}(1 - \|(I-P_{k-1})g_k\|^{-2}a_{k-1}).
 \end{equation}
 If $a_0 > 1$, then this recursion implies that $a_1 < 0$ (since $\|g_k\| \leq 1$ and thus $\|(I-P_{k-1})g_k\| \leq 1$) and as remarked before the desired conclusion is easily seen to hold in this case. Hence we can assume without loss of generality that $a_0 \leq 1$. In addition, if  $a_k \leq 0$ ever holds the result immediately follows for all $n \geq k$. So we also assume without loss of generality in the following that $a_k > 0$.
 
 Utilizing the approximation $\log(1 + x) \leq x$, we rewrite the recursion \eqref{eq-329} as
 \begin{equation}\label{recursion-equation-335}
  \log(a_k) \leq \log(a_{k-1})-\|(I-P_{k-1})g_k\|^{-2}a_{k-1}.
 \end{equation}
 At this point, we could expand the recursion and use that $a_0 \leq 1$ to get
 \begin{equation}
  \log(a_n) \leq -\sum_{k=0}^{n-1}\|(I-P_{k-1})g_k\|^{-2}a_k \leq -a_n\sum_{k=0}^{n-1}\|(I-P_{k-1})g_k\|^{-2},
 \end{equation}
 where the last inequality is due to the fact that the sequence $a_k$ is decreasing. Applying Lemma \ref{entropy-packing-lemma} we obtain
 \begin{equation}
  \log(a_n) \leq -ca_nn^{1+2\delta}.
 \end{equation}
 Solving this gives the desired result up to logarithmic factors.

 In order to remove the logarithmic factors, we use again that the sequence $a_k$ is decreasing and dyadically expand the recursion \eqref{recursion-equation-335} to get
 \begin{equation}
 \begin{split}
  \log(a_{2^j}) &\leq \log(a_{2^{j-1}}) - \sum_{k=2^{j-1}}^{2^j-1} \|(I-P_{k-1})g_k\|^{-2}a_k \\
  &\leq \log(a_{2^{j-1}}) - a_{2^j}\sum_{k=2^{j-1}}^{2^j-1} \|(I-P_{k-1})g_k\|^{-2},
  \end{split}
 \end{equation}
 for each $j \geq 0$ (if we interpret $a_{1/2} = a_0$).
 
 Note that $\|(I-P_{k-1})g_k\| \leq \|(I-\tilde P_{k-1})g_k\|$, where $\tilde P_{k-1}$ is the projection onto the space spanned by $g_{2^{j-1}},...,g_k$. Thus, applying Lemma \ref{entropy-packing-lemma} to the dictionary $\mathbb{D}$ and the sequence $g_{2^{j-1}},...,g_{2^j-1}$, we get
 \begin{equation}
  \log(a_{2^j}) \leq \log(a_{2^{j-1}}) - c 2^{(1+2\delta)(j-1)}a_{2^j}.
 \end{equation}
 Here we may assume, by decreasing $c$ if necessary, that $c \leq 1$.
 Next we will prove by induction that $$a_{2^j} \leq (1+2\delta)c^{-1}2^{-(1+2\delta)(j-1)}.$$
 This completes the proof since $a_k$ is decreasing, and so it implies that
 \begin{equation}
  a_k \leq Ck^{-(1+2\delta)},
 \end{equation}
 with $C = 4^{1+2\delta}(1+2\delta)c^{-1}$.
 
 The proof by induction proceeds as follows. We note that for $j=0$, $a_1 \leq a_0\leq  1 \leq (1+2\delta)c^{-1}2^{1+2\delta}$. (Note that here we use that $c$ is taken $\leq 1$.)
 
 Next, assume that the result holds for $j-1$, so we get
 \begin{equation}\label{logarithmic-recursion}
  \log(a_{2^j}) \leq \log((1+2\delta)c^{-1}) - \log(2)(1+2\delta)(j-2) - c 2^{(1+2\delta)(j-1)}a_{2^j}.
 \end{equation}
 Observe that the left hand side of \eqref{logarithmic-recursion} is an increasing function and the right hand side is a decreasing function of $a_{2^j}$. Further, if we set $a_{2^j} = (1+2\delta)c^{-1}2^{-(1+2\delta)(j-1)}$, we obtain equality in \eqref{logarithmic-recursion}. This implies that 
 $$a_{2^j} \leq (1+2\delta)c^{-1}2^{-(1+2\delta)(j-1)},$$
 completing the inductive step.
\end{proof}

\section{Numerical Experiments}\label{numerical-experiments-section}
In this section, we give numerical experiments which demonstrate the improved convergence rates derived in Section \ref{orthogonal-greedy-section}. The setting we consider is a special case of the situation considered in \cite{lee1996efficient}. We consider the dictionary
\begin{equation}
 \mathbb{P}_0^2:=\{\sigma_0(\omega\cdot x + b):~\omega\in \mathbb{R}^2,~b\in \mathbb{R}\}\subset L^2([0,1]^2),
\end{equation}
where $\sigma_0$ is the Heaviside activation function.
Nonlinear approximation from this dictionary corresponds to approximation by shallow neural networks with Heaviside activation function. The results proved in \cite{siegel2021sharp} imply that the metric entropy of the convex hull of $\mathbb{P}_0^2$ satisfies
\begin{equation}
 \epsilon_n(B_1(\mathbb{P}_0^2)) \eqsim n^{-\frac{3}{4}}.
\end{equation}
We use the orthgonal greedy algorithm to approximate the target function
\begin{equation}
 f(x,y) = \sin(\pi(x+y))^2\sin(\pi(x - y^2))
\end{equation}
by a non-linear expansion from the dictionary $\mathbb{P}_0^2$. Note that the target function $f$ is smooth so that $f\in \mathcal{K}_1(\mathbb{P}_0^2)$ \cite{barron1993universal}. We approximate the $L^2$ norm on the domain $[0,1]^2$ by the empirical $L^2$-norm on a set of $N$ sample points $X_N = \{x_1,...,x_N\}$ drawn uniformly at random from the square $[0,1]^2$. The subproblem
\begin{equation}
 g_k = \arg\max_{g\in \mathbb{P}_0^2} |\langle g,r_k\rangle|
\end{equation}
becomes the combinatorial problem of determining the optimal splitting of the sample points $x_1,...,x_N$ by a hyperplane $\omega\cdot x + b$ such that
\begin{equation}
 |\langle \sigma_0(\omega\cdot x + b),r_k\rangle_{L^2(X_N)}| = \left|\sum_{\omega\cdot x_i + b \geq 0} r_k(x_i)\right|
\end{equation}
is maximized. There are $O(N^2)$ such hyperplane splittings and the optimal splitting can be determined using $O(N^2\log(N))$ operations using a simple modification of the algorithm in \cite{lee1996efficient} which is specific to the case of $\mathbb{R}^2$. 

We run this algorithm with $N=5000$ for a total of $n = 100$ iteration\footnote{All of the code used to run the experiments and generate the plots shown here can be found at \url{https://github.com/jwsiegel2510/OrthogonalGreedyConvergence}}. In Figure \ref{experimental-results} we plot the error as a function of iteration $n$ on a log-log scale. Estimating the convergence order from this plot (here we have removed the first $10$ errors since the rate depends upon the tail of the error sequence) gives a convergence order of $O(n^{-0.717})$. Previous theoretical results for the orthogonal greedy algorithm \cite{devore1996some,barron2008approximation} only give a convergence rate of $O(n^{-\frac{1}{2}})$, while our bounds incorporating the compactness of the dictionary $\mathbb{P}_0^d$ imply a convergence rate of $O(n^{-\frac{3}{4}})$. The empirically estimated convergence order is significantly better than $O(n^{-\frac{1}{2}})$ and is very close to the rate predicted by Theorem \ref{main-theorem}.
\begin{figure}
\begin{center}
 \includegraphics[scale=0.5]{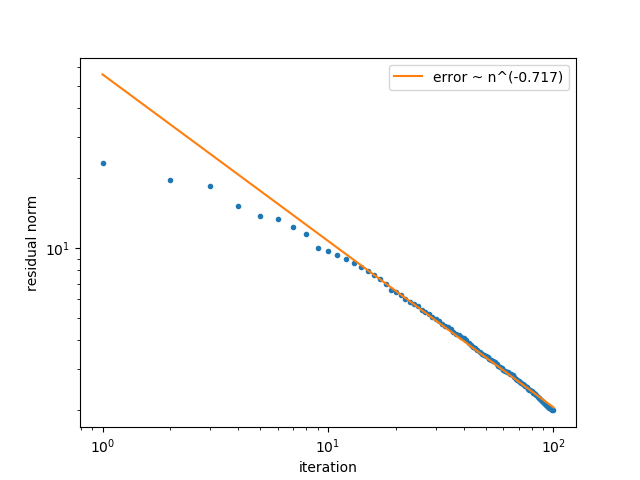}
 \end{center}
 \caption{Estimated convergence rate of the orthogonal greedy algorithm for the dictionary $\mathbb{P}_0^d$. We see that the empirically observed convergence order of $O(n^{-0.717})$ is very close to the rate of $O(n^{-\frac{3}{4}})$ predicted by Theorem \ref{main-theorem}.}
 \label{experimental-results}
\end{figure}

In Figure \ref{function-plots} we plot the target function and the approximants $f_n$ obtained at iterations $5$, $10$, and $100$. This illustrates how the approximation of the target function $f$ generated by the orthogonal greedy algorithm improves as we add more dictionary elements to our expansion.
\begin{figure}
 \begin{center}
 \includegraphics[scale=0.5]{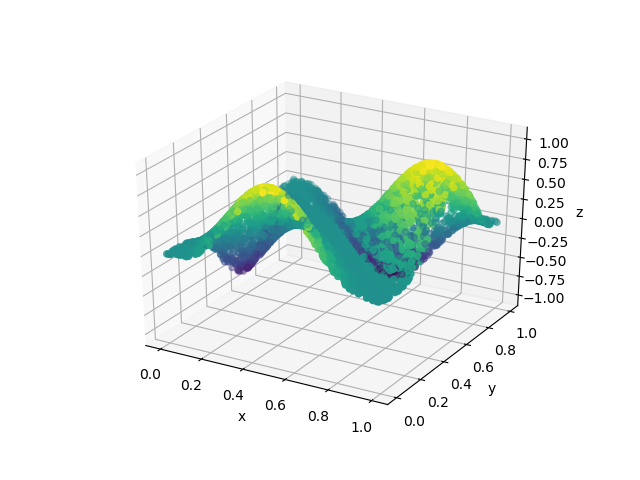}
 \includegraphics[scale=0.5]{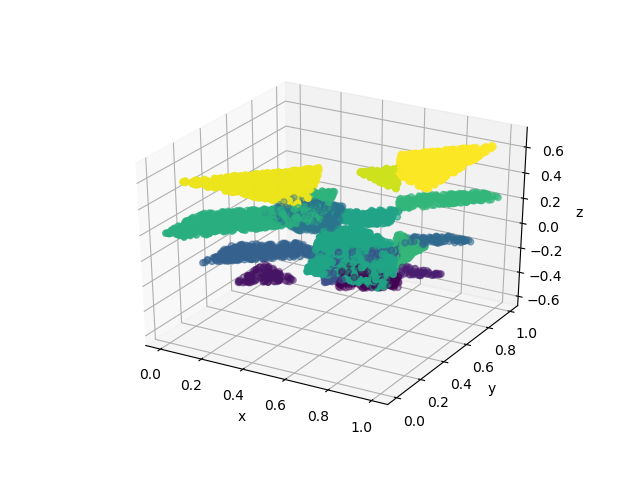}
 \includegraphics[scale=0.5]{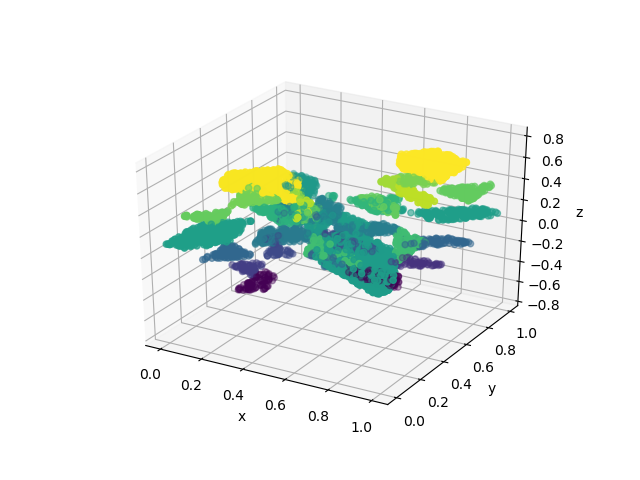}
 \includegraphics[scale=0.5]{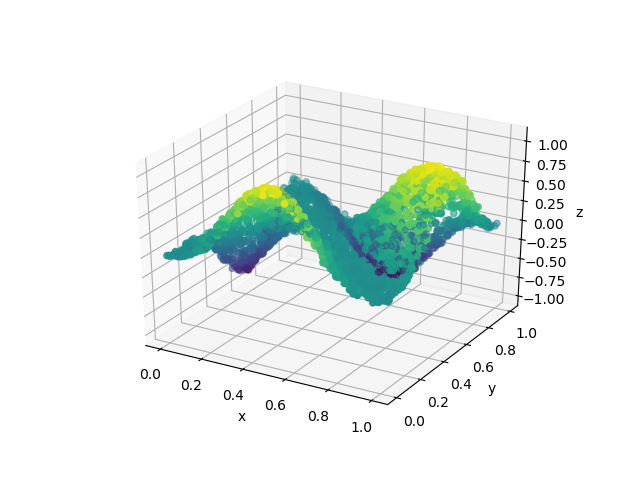}
 \end{center}
 \caption{(Top Left) Target function $f$, (Top Right) Approximation after $5$ iterations of the orthogonal greedy algorithm, (Bottom Left) Approximation after $10$ iterations of the orthogonal greedy algorithm, (Bottom Right) Approximation after $100$ iterations of the orthogonal greedy algorithm.}
 \label{function-plots}
\end{figure}

\section{Lower Bounds}\label{lower-bounds-section}
In this section, we show that the the iterates generated by the orthogonal greedy algorithm cannot be bounded in $\mathcal{K}_1(\mathbb{D})$ and that the rate derived in Theorem \ref{main-theorem} cannot be further improved.

\begin{proposition}\label{no-l1-bound-proposition}
 For each $R < \infty$, there exists a (normalized) dictionary $\mathbb{D}\subset H$ and an $f\in B_1(\mathbb{D})$ such that if $f_3$ is the iterate generated by the orthogonal greedy algorithm when applied to $f$ after $3$ steps, then
 \begin{equation}
  \|f_3\|_{\mathcal{K}_1(\mathbb{D})} > R.
 \end{equation}

\end{proposition}
This shows that in the worst case the iterates generated by the orthogonal greedy algorithm are not bounded in the variation space $\mathcal{K}_1(\mathbb{D})$.
\begin{proof}
 Let $H = \mathbb{R}^5$ with orthonormal basis $e_1,...,e_5$. Let $0 < \delta < \epsilon < \frac{1}{2}$ and consider the dictionary $\mathbb{D} = \{x_1,...,x_5\}\subset H$ given by
 \begin{equation}
 \begin{split}
  x_1 &= \epsilon e_1 - \sqrt{1-\epsilon^2}e_2 \\ 
  x_2 &= \epsilon e_2 + \sqrt{1-\epsilon^2}e_3 \\
  x_3 &= e_3 \\
  x_4 &= \frac{\epsilon}{4}(e_1 + e_2) + \frac{1}{2}e_3 + ce_4 + \delta e_5 \\
  x_5 &= \frac{\epsilon}{4}(e_1 + e_2) + \frac{1}{2}e_3 - ce_4 + \delta e_5,
  \end{split}
 \end{equation}
 where $c$ is chosen so that $\|x_i\|_H = 1$ for all $i$. Consider the element
 \begin{equation}
  f = \frac{\epsilon}{4}(e_1 + e_2) + \frac{1}{2}e_3 + \delta e_5 = \frac{1}{2}x_4 + \frac{1}{2}x_5\in B_1(\mathbb{D}).
 \end{equation}
 We claim that if $\delta < \frac{\epsilon}{\sqrt{8}}$, then the orthogonal greedy algorithm applied to $f$ and $\mathbb{D}$ will select $g_1 = x_3$, $g_2 = x_2$, and $g_3 = x_1$ in the first $3$ steps. 
 
 Indeed, we calculate
 \begin{equation}
 \begin{split}
  \langle f,x_1\rangle = \frac{\epsilon}{4}(\epsilon - \sqrt{1 - \epsilon^2}),~\langle f,x_2\rangle &= \frac{1}{2}\left(\frac{\epsilon^2}{2} + \sqrt{1 - \epsilon^2}\right),~\langle f,x_3\rangle = \frac{1}{2},\\
  \langle f,x_4\rangle = \langle f,x_5\rangle &= \frac{\epsilon^2}{8} + \frac{1}{4} + \delta^2.
  \end{split}
 \end{equation}
 We now verify, by differentiating for example, that for $\epsilon > 0$ we have
 \begin{equation}
  \frac{\epsilon^2}{2} + \sqrt{1 - \epsilon^2} < 1,
 \end{equation}
 which implies that $|\langle f,x_2\rangle| < |\langle f,x_3\rangle|$. The other inequalities are obvious (recalling that $\delta < \frac{\epsilon}{\sqrt{8}}$) and so it holds that $g_1 = x_3$. Projecting $f$ orthogonal to $x_3 = e_3$, we get $$r_1 := f - f_1 = \frac{\epsilon}{4}(e_1 + e_2) + \delta e_5.$$ Calculating inner products, we see that
 \begin{equation}
  \begin{split}
  \langle r_1,x_1\rangle &= \frac{\epsilon}{4}(\epsilon - \sqrt{1 - \epsilon^2}),~\langle r_1,x_2\rangle = \frac{\epsilon^2}{4},\\
  \langle f,x_4\rangle &= \langle f,x_5\rangle = \frac{\epsilon^2}{8} + \delta^2.
  \end{split}
 \end{equation}
 Since $\delta < \frac{\epsilon}{\sqrt{8}}$, these relations imply that $g_2 = x_2$. Projecting orthogonal to $\text{span}(g_1,g_2) = \text{span}(x_3,x_2) = \text{span}(e_2,e_3)$, we get $$r_2 := f - f_2 = \frac{\epsilon}{4}e_1 + \delta e_5.$$
 Finally, computing inner products, we get
 \begin{equation}
  \langle r_1,x_1\rangle = \frac{\epsilon^2}{4},~\langle f,x_4\rangle = \langle f,x_5\rangle = \frac{\epsilon^2}{16} + \delta^2.
 \end{equation}
 Again, since $\delta < \frac{\epsilon}{\sqrt{8}}$ this implies that $g_3 = x_1$. Projecting orthogonal to $\text{span}(g_1,g_2,g_3) = \text{span}(e_1,e_2,e_3)$, we obtain
 \begin{equation}
  r_3 = \delta e_5~\text{and}~f_3 = \frac{\epsilon}{4}(e_1 + e_2) + \frac{1}{2}e_3.
 \end{equation}
 Now, if $f_3 = \sum_{i=1}^5 a_ix_i$, then by taking inner products with $e_4$ and $e_5$, we see that $a_4 - a_5 = 0$ and $a_4 + a_5 = 0$ which implies that $a_4 = a_5 = 0$. Thus $f_3 = \sum_{i=1}^3 a_ix_i$ and since the $x_i$ are linearly independent, the $a_i$ are uniquely determined, and a simple calculation shows them to be
 \begin{equation}
  a_1 = \frac{1}{4},~a_2 = \frac{1}{4} + \frac{\sqrt{1-\epsilon^2}}{\epsilon},~a_3 = \frac{1}{2} - \frac{\sqrt{1-\epsilon^2}}{4}-\frac{1-\epsilon^2}{\epsilon}.
 \end{equation}
 We finally get
 \begin{equation}
  \|f\|_{\mathcal{K}_1(\mathbb{D})} = |a_1| + |a_2| + |a_3| \geq \frac{\sqrt{1-\epsilon^2}}{\epsilon}.
 \end{equation}
 Letting $\epsilon\rightarrow 0$, we obtain the desired result.

\end{proof}

\begin{proposition}
 There exists a Hilbert space $H$ and a dictionary $\mathbb{D}\subset H$ such that $\epsilon_n(B_1(\mathbb{D})) \leq Cn^{-\frac{1}{2}-\alpha}$ and for each $n$
 \begin{equation}
  \sup_{f\in B_1(\mathbb{D})} \|f - f_n\|_H \geq Kn^{-\frac{1}{2}-\alpha}, 
 \end{equation}
 where $f_n$ is the $n$-th iterate of the orthogonal greedy algorithm applied to $f$.
\end{proposition}
This implies the optimality of the rates in Theorem \ref{main-theorem} under the fiven assunmptions on the entropy of $B_1(\mathbb{D})$.

\begin{proof}
Let $H = \ell^2$ and consider the dictionary
\begin{equation}
 \mathbb{D} = \{k^{-\alpha}e_k:~k\geq 1\}.
\end{equation}
It is known that $\epsilon_n(B_1(\mathbb{D})) \leq Cn^{-\frac{1}{2} - \alpha}$ for a constant $C$ \cite{ball1990entropy}. Let $N$ be an integer, and consider the element
\begin{equation}
 f_N = \frac{1}{N}\sum_{k=1}^N k^{-\alpha}e_k.
\end{equation}
We obviously have that $\|f_N\|_{\mathcal{K}_1(\mathbb{D})} = 1$. Moreover, it is clear that after $n$ iterations of the orthogonal greedy algorithm, the residual $r_n$ will satisfy
\begin{equation}
 r_n = \frac{1}{N}\sum_{k=n+1}^N k^{-\alpha}e_k,
\end{equation}
since the dictionary elements $g_k$ chosen at each iteration will be $g_k = k^{-\alpha}e_k$. Choosing $N = 2n$, we get
\begin{equation}
 \|r_n\| = \left(\frac{1}{4n^2}\sum_{k=n+1}^{2n} k^{-2\alpha}\right)^{\frac{1}{2}} \geq \left(\frac{1}{4n}(2n)^{-2\alpha}\right)^{\frac{1}{2}} \geq 2^{-(1+\alpha)}n^{-\frac{1}{2}-\alpha}.
\end{equation}
\end{proof}
This shows that the rate derived in Theorem \ref{main-theorem} is unimprovable. To conclude, we note that in the preceding argument we used a sequence of elements $f_N\in B_1(\mathbb{D})$. However, the same lower bound can be obtained for a single element $f$ by modifying the construction in \cite{livshits2009lower}, Section 8 in a straightforward manner.

\section{Conclusion}
We have shown that the orthogonal greedy algorithm achieves an improved convergence rate on dictionaries whose convex hull is compact. An important point, however, is that the expansions thus derived generally do not have their coefficients $a_i$ bounded in $\ell^1$. It is an important follow-up question whether the improved rates can be obtained by a greedy algorithm which satisfies this further restriction. Another interesting research direction is to extend our analysis to greedy algorithms for other problems such as reduced basis methods \cite{devore2013greedy,binev2011convergence} or sparse PCA \cite{cai2013sparse}.

\section{Acknowledgements}
We would like to thank Professors Russel Caflisch, Ronald DeVore, Weinan E, Albert Cohen, Stephan Wojtowytsch and Jason Klusowski for helpful discussions. We would also like to thank the anonymous referees for their helpful comments. This work was supported by the Verne M. Willaman Chair Fund at the Pennsylvania State University, and the National Science Foundation (Grant No. DMS-1819157 and DMS-2111387).

\bibliographystyle{spmpsci}
\bibliography{refs}

\end{document}